\documentclass[12pt,reqno]{amsart}
\usepackage[pdftex]{hyperref}
\usepackage{amsmath,amsthm,amsopn,amssymb,a4wide,times,varioref}
\usepackage{enumitem}
\usepackage[mathscr]{eucal}
\usepackage[utopia]{mathdesign}
\usepackage{tikz}
\usepackage[numbers,sort&compress]{natbib}
\def\circleds{{\bigcirc \!\!\!\!\!\!\!\!s}\;\;}
\newtheorem{theorem}{Theorem}[section]
\newtheorem{lemma}[theorem]{Lemma}
\newtheorem{corollary}[theorem]{Corollary}

\newtheorem{definition}[theorem]{Definition}

\numberwithin{equation}{section}
\setlist{label={$($\roman{enumi}\kern1pt$)$}}
\begin{document}
\title{The Defect sequence for contractive tuples}
\thanks{This research is supported by UGC SAP and UKIERI project Quantum Probability, Noncommutative
 Geometry and Quantum Information}%
\author[Bhattacharyya]{Tirthankar Bhattacharyya}

\address{
 (Bhattacharyya) Department of Mathematics\\
  Indian Institute of Science\\
  Bangalore\\
  560 012\\
  India}

\email{tirtha@math.iisc.ernet.in}

\author[Das] {Bata Krishna Das}

\address{
(Das) Department of Mathematics and Statistics\\
Lancaster University\\
Lancaster LA1 4YF\\
United Kingdom }

\email{b.das@lancaster.ac.uk}

\author[Sarkar]{Santanu Sarkar}

\address{
(Sarkar)  Department of Mathematics\\
  Indian Institute of Science\\
  Bangalore\\
  560 012\\
  India}

\email{santanu@math.iisc.ernet.in}

\begin{abstract}
We introduce the defect sequence for
a contractive tuple of Hilbert space operators and investigate its properties.
The defect sequence is a sequence of numbers, called defect dimensions associated with a contractive tuple.
We show that there are upper bounds for the defect dimensions. The tuples for which these upper bounds are obtained, are called maximal contractive tuples. The upper bounds are different in the non-commutative and in the commutative case.
We show that the creation operators on the full Fock space and the co ordinate multipliers on the Drury-Arveson space are maximal.
We also study pure tuples and see how the defect dimensions play a role in their irreducibility.
\end{abstract}


\maketitle

\section{Introduction}

Let us fix a positive integer $d$ for this paper.  Consider a tuple $(T_1,\cdots,T_d)$ of bounded operators on a complex,separable,infinite dimensional Hilbert space $H$ and  the completely positive map $P_T : \mathcal B(  H) \rightarrow \mathcal B(  H)$ given by $$P_T(X) = \sum_{i=1}^d T_iXT_i^* , X \in \mathcal B(H) $$ We assume that  $$T_1T_1^* + \cdots+T_dT_d^* \leq I.$$ Such a tuple
$T = (T_1, T_2, \ldots ,T_d)$ is called a \textit{ row contraction }\cite{Pop-Poisson}
or a \textit{contractive tuple} \cite{{BES},{Pop-Poisson}}. Arveson has called it \textit{$d$ contraction} \cite{sub3} in case $T_iT_j = T_jT_i$ for all $i,j = 1,\cdots,d.$ The completely positive map $P_T$ plays a crucial role in dilation theory of a row contraction, see \cite {Bhattacharyya-survey}. The contractivity assumption above means that $P_T$ is a completely positive and contractive map. Thus $$I \geq P_{T} (I) \geq P_{T} ^2(I) \geq \cdots.$$This sequence of positive operators is decreasing and hence has a strong limit. If the limit is $0,$ then the tuple $T$ is called \textbf{ pure}.
In this note, we associate a sequence of numbers which we shall call the defect  dimensions  with such an operator tuple.  Let $H_1 =  \overline{\mbox{ Range }} (I - P_T(I) )$ be the first defect space, and let $\Delta _T = \mbox {dim} H_1$ be the \textit{first defect dimension.}\\ Given an $m$-tuple $B = (B_1, B_2, \ldots ,B_m)$ and a $k$-tuple $C = (C_1, C_2, \ldots ,C_k)$ of bounded operators on $H$, define their product to be the $mk$-tuple
\begin{equation} \label{product} B C = (B_1 C_1, \ldots ,B_1 C_k , B_2 C_1, \ldots B_2 C_k, \ldots , B_i C_j , \ldots , B_m C_1, \ldots , B_m C_k ). \end{equation} This product has been very useful in proving the spectral radius formula, See \cite{BB}.
With this definition, $T^n$ is a $d^n$-tuple whose typical entry is $T_{i_1} T_{i_2} \ldots T_{i_n}$ with $1 \le i_1, i_2, \ldots ,i_n \le d$. Note that $P_T^n(I) = P_{T^n}(I)$. Let $H_n = \overline{\mbox{ Range }} (I - P_T^n(I) )$ be the $n-$th defect space, and let $\Delta_{T^n} = \mbox{dim} H_n$ be the \textit{$n-$th defect dimension.} Thus if $H_1 = 0$ then $H_n =0$ for all $n\geq 1$. So we assume that $0< \Delta _{T} < \infty.$ Then $\parallel P_{T} (I) \parallel < 1$ would imply that $I - P_{T} (I)$ is invertible and $H_1 = H.$ Since we have assumed that $H$ is infinite dimensional,  so the above implies that $\parallel P_{T} (I) \parallel = 1.$ This study is motivated by \cite {DC} where these questions were investigated for $d = 1.$ In section $2,$ we find upper bounds for the defect dimensions, show that the defect sequence is non-decreasing and investigate stability of the defect sequence. In this section we also show that for a pure tuple, the defect sequence is strictly increasing. In section $3,
$ we investigate the maximal non-commuting contractive tuples. Section $4$ deals with commuting ones and section $5$ with the pure ones.\\
\textbf {Any tuple $T$ considered in this note is contractive and satisfies $0< \Delta _{T} < \infty.$ }

\section{Properties of defect dimensions}
\begin{lemma}
 The sequence $\{ \Delta_{T^n} : n=1,2,\cdots \}$ is increasing in $n$  and
 $ \Delta_{T^n} \le (1 + d + \cdots +d^{n-1}) \Delta _T$ for each positive integer $n$ .\end{lemma}
 \begin{proof}
 First note that, $I \geq P_T(I) \geq P_T^2(I)\geq \cdots$
so that $$0 \leq I - P_T(I) \leq  I - P_T^2(I) \leq \cdots.$$ So  $0 < \Delta _T \leq \Delta _{T^2} \leq \cdots$. This shows that the sequence $\{ \Delta_{T^n} : n=1,2,\cdots \}$ is increasing in $n$.
Now, note that if $A = BC$ as in (\ref{product}) above, where $B$ and $C$ are contractive tuples, then
\begin{equation} \label{CandA} \Delta_B \le \Delta_A. \end{equation} Indeed,
\begin{equation} \label{BandC} I - P_{BC(I)} =  I - P_B(P_C(I)) =  I - P_B(I) + P_B(I) - P_B(P_C(I))
= I - P_B(I) + P_B(I  - P_C(I))  \end{equation}
This implies that $ I - P_{BC(I)} \ge I - P_B(I)$ and hence $\Delta_{BC} \ge \Delta_B$. We are done. Secondly,
\begin{equation} \label{dbcledc+kdb}\Delta_{BC} \le \Delta_B + m \Delta_C. \end{equation} This again follows from (\ref{BandC}). The equation (\ref{BandC}) implies that
\begin{equation} \label{BandCagain}  \mbox{ Range } (I - P_{BC}(I) ) \subseteq \mbox{ Range } (I - P_{B}(I) ) + \mbox{ Range } (P_B(I - P_{C}(I)) ). \end{equation}
The second term on the right side is contained in $  \sum_{i=1}^m $ Range $B_i ( I - P_C(I) ) B_i^*$. So dim $ \overline{\mbox{ Range }} (P_B(I - P_C(I)) \le m \Delta_C$ and so taking dimensions of both sides in (\ref{BandCagain}), we have (\ref{dbcledc+kdb}).

We are ready to finish the proof. Applying ( \ref{dbcledc+kdb}) to $B=C=T$,  we get
$$ \Delta_{T^2} \le \Delta_T + d \Delta_T = (d+1) \Delta_T.$$ Similarly, applying ( \ref{dbcledc+kdb}) with $B=T$ and $C=T^2$, we get $$\Delta_{T^3} \le \Delta_T + d \Delta_{T^2} \le \Delta_T + d(d+1) \Delta_T = (d^2 + d + 1) \Delta_T.$$
Now a simple induction argument gives that $\Delta_{T^n} \le (d^{n-1} + d^{n-2} + \cdots + d + 1 ) \Delta_T $.
 \end{proof}

We shall give a description of $H_n$ and thereby we shall give a result concerning the stability of the sequence $\Delta_{T^n}$. For that, we need to develop some notations.
Given a contractive tuple $T = (T_1, T_2, \cdots ,T_d)$,~let $(T,1):H \bigoplus \cdots \bigoplus H = H^d \rightarrow H$ be  given by $$(T,1) \begin{pmatrix}
x_1\\x_2\\ \vdots\\ x_d
\end{pmatrix} = T_1x_1+\cdots+T_dx_d.$$ So $(T,1)^*:H \rightarrow H^d$ is given by $$(T,1)^*x = \begin{pmatrix}
T_1^*x\\ \vdots\\T_d^* x
\end{pmatrix}.$$ and $(T,1)(T,1)^*: H \rightarrow H$ is given by $(T,1)(T,1)^*x = P_T(I)x$. Now,$$(T,d):\underbrace{H^d\bigoplus \cdots \bigoplus H^d}_d = H^{d^2} \longrightarrow H^d$$ is defined by $$(T,d) \begin{pmatrix}
x_1\\ \vdots \\ x_d\\ x_{d+1}\\ \vdots \\x_{2d}\\ \vdots \\ \vdots \\x_{(d-1)d} \\ \vdots \\x_{d^2}
\end{pmatrix} = \begin{pmatrix}(T,1) \begin{pmatrix}
x_1\\x_2\\ \vdots\\ x_d
\end{pmatrix}\\ \vdots \\ \vdots \\ (T,1)\begin{pmatrix}x_{(d-1)d} \\ \vdots \\x_{d^2}
\end{pmatrix}
\end{pmatrix}$$
Thus $(T,d)^* : H^d \rightarrow H^{d^2}$ satisfies $$(T,d)^* \begin{pmatrix}
x_1\\x_2\\ \vdots\\ x_d
\end{pmatrix} = \begin{pmatrix}
(T,1)^*x_1\\ \vdots\\ (T,1)^*x_d
\end{pmatrix}$$
and $(T,d)(T,d)^* : H^d \rightarrow H^d$ is given by $$(T,d)(T,d)^* \begin{pmatrix}
x_1\\x_2\\ \vdots\\ x_d
\end{pmatrix} = \begin{pmatrix}
(T,1)(T,1)^*x_1\\ \vdots\\(T,1)(T,1)^* x_d
\end{pmatrix}. $$ Similarly we can define $(T,d^k) : H^{d^{(k+1)}} \rightarrow H ^{d^k}$ for $k= 1,2,\cdots .$

 Now consider the $d^2$ length tuple $T^2.$  Then $(T^2,1) : H^{d^2} \rightarrow H$ is defined by $$(T^2,1)\begin{pmatrix}
x_1\\x_2\\ \vdots\\ x_{d^2}
\end{pmatrix} = T_1^2 x_1+ \cdots +T_1 T_d x_d + \cdots+ T _d^2 x_{d^2} $$ and $(T^2,1)^*: H \rightarrow H^{d^2}$ is given by  $$(T^2,1)^* x = \begin{pmatrix} T_1^{2^*}x\\ \vdots\\ T_d^{2^*}x \end{pmatrix}.$$  $(T^2,1) (T^2,1)^* : H \rightarrow H $ satisfies $$(T^2,1) (T^2,1)^* x =P_T^2(I)x .$$  $(T^2,d):\underbrace{H^{d^2}\bigoplus \cdots \bigoplus H^{d^2}}_d = H^{d^3} \longrightarrow H^d $ is defined by $$(T^2 , d) \begin{pmatrix} \underline{x_1} \\ \vdots \\ \underline{x_d}
\end{pmatrix} = \begin{pmatrix}
(T^2,1)\underline{x_1} \\ \vdots \\ (T^2,1)\underline{x_d}
\end{pmatrix}$$ where ${\underline{x_1},\cdots,\underline{x_d}} \in H^{d^2}$ and $(T^2,d)^* : H^d \rightarrow H^{d^3} $ is given by $$(T^2,d)^* \begin{pmatrix}
x_1\\x_2\\ \vdots\\ x_d
\end{pmatrix} = \begin{pmatrix}(T^2,1)^* x_1 \\\vdots \\(T^2,1)^*x_d
\end{pmatrix}.$$ Similarly, we can define $(T^2,d^k) : H^{d^{k+2}} \rightarrow H^{d^k}$ for $k= 1,2,\cdots .$ In the same fashion, we can define $$(T^k,1) : H^{d^k} \rightarrow H.$$
\begin{theorem}
The defect spaces $H_n$ have the following properties :\begin{enumerate} \item $H_n \subseteq H_{n+1}$  for all $n=1 , 2 ,\cdots.$
\item $H_n = H_1 \vee (T,1) A^d \vee (T^2,1) A^{d^2} \vee \cdots \vee (T^{n-1},1) A^{d^{n-1}}$ for all $n= 2,3,\cdots$ where $$A^{d^k} = \overline{\mbox {Range}}(I_{H^{d^k}} - (T ,d^k) (T,d^k)^*) \subseteq H^{d^k}$$ and $(T^k ,1)A^{d^{k}} \subseteq H_{k+1}$  for all $ k=1,2,\cdots. $
\end{enumerate}
\end{theorem}
\begin{proof}
$(i)$ First observe that $$(T^{n+1},1) = (T^n,1) (T,d^n)~ \mbox{and}~H_n = \overline{\mbox{ Range }} (I - (T^n,1)(T^n ,1)^* ).$$ Note  $x \in \mbox{ ker} (I_H - (T^{n+1},1)(T^{n+1},1)^*) ~\mbox{if and only if}~ \parallel (T^{n+1} ,1)^* x\parallel^2 = \parallel x \parallel^2 .$ Take $$ x \in \mbox{ ker} (I_H - (T^{n+1},1)(T^{n+1},1)^*).$$ Therefore $$\parallel (T^{n+1} ,1)^* x\parallel^2 = \parallel x \parallel^2.$$ So $$ \parallel x \parallel^2 = \parallel (T^{n+1} ,1)^* x\parallel^2
\leq \parallel (T^n ,1)^* x\parallel^2 \leq \cdots \leq \parallel (T ,1)^* x\parallel^2 \leq \parallel x \parallel^2 .$$
Thus we have $$ \parallel x \parallel^2 = \parallel (T^{n+1} ,1)^* x\parallel^2 = \parallel (T^n ,1)^* x\parallel^2
 = \cdots = \parallel (T ,1)^* x\parallel^2. $$
The above  implies $$x \in \mbox{ ker} (I_H - (T^n,1)(T^n ,1)^*).$$ Hence $ \mbox{ ker}(I_H - (T^n,1)(T^n ,1)^*)^\perp \subseteq  \mbox{ ker} (I_H - (T^{n+1},1)(T^{n+1} ,1)^*)^ \perp$. As a result $$H_n \subseteq H_{n+1}~\mbox{ for all}~ n=1 , 2 ,\cdots.$$

$(ii)$  Note that $A^{d^k} = H_1 \oplus \cdots \oplus H_1 $ ($d^k$ copies)  for $k= 1,2,\cdots .$\\  Take $$x \in \mbox {Range} (I_H - (T^n,1)(T^n ,1)^*). $$ which implies $$ x = (I_H - (T^n,1)(T^n ,1)^*)y.$$ Thus we have \begin{align*} x &= [I_H - (T,1)(T ,1)^*]y +[(T,1)(T ,1)^* -(T^2,1)(T^2 ,1)^*]y +[(T^2,1)(T^2 ,1)^* - (T^3,1)(T^3 ,1)^*]y\\ &+ \cdots   + [(T^{n-1},1)(T^{n-1} ,1)^* - (T^n,1)(T^n ,1)^*]y .
\end{align*}
So, \begin{align*} x&= [I_H - (T,1)(T ,1)^*]y + (T,1)[I_{H^d} - (T,d)(T,d)^*] (T,1)^*y + (T^2,1) [I_{H^{d^2}} - (T,d^2) (T,d^2)^*] (T^2,1)^*y\\ &+ \cdots + (T^{n-1},1)[I_{H^{d^{n-1}}} - (T,d^{n-1})(T,d^{n-1})^*](T^{n-1},1)^*y.\end{align*}

Hence, $$H_n \subseteq H_1 \vee (T,1) A^d \vee (T^2,1) A^{d^2} \vee \cdots \vee (T^{n-1},1) A^{d^{n-1}}~ \mbox{for all}~ n= 2,3,\cdots$$ where $A^{d^k} = \overline{\mbox {Range}}(I_{H^{d^k}} - (T ,d^k) (T,d^k)^*) \subseteq H^{d^k}$.  For the converse, we shall show that $$(T^k ,1)A^{d^{k}} \subseteq H_{k+1}~\mbox{  for all}~  k=1,2,\cdots. $$

Take $$x \in \mbox{ ker} (I_H - (T^{k+1},1)(T^{k+1},1)^*).$$  Therefore $$\parallel (T^{k+1} ,1)^* x\parallel^2 = \parallel x \parallel^2.$$ So $$ \parallel x \parallel^2 = \parallel (T^{k+1} ,1)^* x\parallel^2
= \parallel (T ,d^k)^* (T^k,1)^*x \parallel^2 \leq \parallel (T^k,1)^*x \parallel^2 \leq \cdots \leq  \parallel x \parallel^2 .$$ Which implies $$\parallel (T ,d^k)^* (T^k,1)^*x \parallel^2 = \parallel (T^k,1)^*x \parallel^2$$

Thus $$(T^k,1)^*x \in \mbox{ ker} (I_{H^{d^k}}- (T,d^k)(T,d^k)^*).$$ The above implies $(T^k,1)^*$ maps $\mbox{ker} (I_H - (T^{k+1},1)(T^{k+1},1)^*)$ to $\mbox{ker} (I_{H^{d^k}}- (T,d^k)(T,d^k)^*)$ isometrically for each $k \geq 0.$
Therefore $(T^k,1)$ maps $A^{d^k}$ to $H_{k+1}.$
 As a result we have $$H_n = H_1 \vee (T,1) A^d \vee (T^2,1) A^{d^2} \vee \cdots \vee (T^{n-1},1) A^{d^{n-1}}~ \mbox{ for all}~ n= 2,3,\cdots.$$
\end{proof}

\begin{corollary} $H_n= H_1 \vee_{i_{1}=1} ^{d} T_{i_{1}}{H_1} \vee_{i_1,i_2=1}^d T_{i_1}T_{i_2}H_1 \vee \cdots \vee_{i_1,\cdots,i_{n-1} = 1}^{d} T_{i_{1}} \cdots T_{i_{n-1}} H_1 $  for all $n\geq 2.$
\end{corollary}
\begin{proof}
Directly follows from Theorem 2.2 and the fact that $A^{d^k} = H_1 \oplus \cdots \oplus H_1 $ ($d^k$ copies)  for $k= 1,2,\cdots .$
\end{proof}

The following theorem shows that the defect sequence either is strictly increasing or stabilises after finitely many steps.
\begin{theorem}If $\Delta _{T^n} = \Delta _{T^{n+1}} \mbox {for some n}, $ then $\Delta _{T^k} = \Delta _{T^n}$ for all $k \geq n .$
 \end{theorem}
 \begin{proof}
 $\Delta _{T^n} = \Delta _{T^{n+1}}$ implies that $H_n = H_{n+1}.$ We shall show that $H_{n+2} = H_{n+1}.$ Now $H_n = H_{n+1}$ implies that $$(T^n,1) A^{d^n} \subseteq H_1 \vee (T,1) A^d \vee (T^2,1) A^{d^2} \vee \cdots \vee (T^{n-1},1) A^{d^{n-1}}.$$ To show $H_{n+1} = H_{n+2}$ we shall show that $$(T^{n+1},1) A^{d^{n+1}} \subseteq (T,1) A^d \vee (T^2,1) A^{d^2} \vee \cdots \vee (T^n,1) A^{d^n}.$$ Then it will follow that $H_{n+2} \subseteq H_{n+1}.$ From that it follows that $H_{n+1} = H_{n+2}.$

\small{\begin{align*} (T^{n+1},1) A^{d^{n+1}}& = (T,1)\begin{pmatrix}
(T^n,1) A^{d^n}\\ \vdots\\(T^n,1) A^{d^n}
\end{pmatrix}\\ &\subseteq (T,1) \begin{pmatrix}H_1 \vee (T,1) A^d \vee (T^2,1) A^{d^2} \vee \cdots \vee (T^{n-1},1) A^{d^{n-1}}\\ \vdots \\H_1 \vee (T,1) A^d \vee (T^2,1) A^{d^2} \vee \cdots \vee (T^{n-1},1) A^{d^{n-1}} \end {pmatrix} \\ &\subseteq (T,1) \begin{pmatrix}
H_1\\ \vdots\\H_1
\end{pmatrix} \vee (T,1) \begin{pmatrix}
(T,1) A^{d}\\ \vdots\\(T,1) A^{d}
\end{pmatrix} \vee \cdots \vee (T,1) \begin{pmatrix}
(T^{n-1},1) A^{d^{n-1}}\\ \vdots\\(T^{n-1},1) A^{d^{n-1}}
\end{pmatrix}\\& = (T,1) A^d \vee (T^2,1) A^{d^2} \vee \cdots \vee (T^n,1) A^{d^n}. \end{align*}}
  \end{proof}

 Here is a condition to ensure that the defect sequence is strictly increasing.
 \begin{lemma}
For a pure tuple $T,$ $\Delta_{T^n} <\Delta_{T^{n+1}}  $ for all $n.$
 \end{lemma}

\begin{proof}
Let there exist $k$ such that $\Delta_{T^k} = \Delta_{T^{k+1}}. $ Then from theorem 2.3 we have $\Delta_{T^n} = \Delta_{T^k} $ for all $n \geq k.$ Which implies $H_k = H_{k+1} = \cdots.$ As a result we have $$\mbox{ ker} (I_H - (T^k,1)(T^k,1)^*) = \mbox { ker} (I_H - (T^{k+1},1)(T^{k+1},1)^* = \cdots.$$ So taking $x\in \mbox{ ker} (I_H - (T^k,1)(T^k,1)^*)$ we have $$\parallel x \parallel = \parallel (T^k,1)^*x \parallel = \parallel (T^{k+1},1)^*x \parallel = \cdots . $$ Note that purity of $T$ is the same as $(T^n,1)^*$ converging to $0$ strongly. Since  $(T^n,1)^*$ strongly goes to zero, we have  $x=0.$ So we have $ \mbox {ker} (I_H - (T^k,1)(T^k,1)^*) = \{0\}.$ Which implies $H_k = H.$ So we have $\Delta _{T^k} = \mbox{dim} H = \infty ,$which is a contradiction. So we have $\Delta_{T^n} <\Delta_{T^{n+1}}  $ for all $n.$
 \end{proof}

Now we shall see a condition for which $\Delta_{T^n} = n.$

\begin{lemma}
If $\mbox{ dim} (H_1) =1, $  $\mbox{dim} {(T^n,1)A^{d^n} } = 1$ and $T$ is pure, then $\Delta_{T^n} = n.$
\end{lemma}
\begin{proof}  We have seen that $H_n \subseteq H_1 \vee (T,1) A^d \vee (T^2,1) A^{d^2} \vee \cdots \vee (T^{n-1},1) A^{d^{n-1}}$ for all $n= 2,3,\cdots$ where $$A^{d^k} = \overline{\mbox {Range}}(I_{H^{d^k}} - (T ,d^k) (T,d^k)^*) \subseteq H^{d^k}.$$
Hence $ \Delta _{T^n} \leq n$ for all $n = 2,3,\cdots.$
If there exists $n_0$ such that $\Delta _{T^{n_0}} < n_{0},$then since $\Delta _T = 1$ and $\{\Delta _{T^n}\}$ is increasing, we have $\Delta _{T^{n_0 - 1}} = \Delta _{T^{n_0}},$ which can not happen from the previous lemma. So we have $\Delta_{T^n} = n.$
 \end{proof}

\begin{theorem}
$\Delta _{T^n} = dim (\overline{\mbox{Range}} (I_{H^{d^n}} - (T^n,1)^* (T^n,1))$ if and only if  $\mbox{dim} ( \mbox {ker} (T^n,1)) = \mbox  {dim }(\mbox{ ker }(T^n,1)^*).$
\end{theorem}
\begin{proof}
First we shall show that if $ \mbox{dim} (\mbox {ker} (T^n,1)) = \mbox{ dim} ( \mbox {ker} (T^n,1)^*)$ then there exists a unitary operator $U : H^{d^n} \rightarrow  H$ such that $(T^n , 1) = U ((T^n,1)^* (T^n,1)) ^{1/2}.$ First observe  $$ \parallel ((T^n,1)^* (T^n,1)) ^{1/2} x \parallel ^2 = \parallel (T^n,1)x \parallel ^2 . $$ The operator $V :\overline{\mbox {Range}} ((T^n,1)^* (T^n,1)) ^{1/2} \rightarrow \overline{\mbox {Range}} (T^n,1)  $ is an onto isometry. Also $$\overline{\mbox {Range}} ((T^n,1)^* (T^n,1)) ^{1/2} = \overline{\mbox {Range}} (T^n , 1) ^*. $$ So $V : \overline{\mbox {Range}} (T^n , 1) ^* \rightarrow \overline{\mbox {Range}} (T^n , 1)  $ is unitary and $ (T^n,1) = V ((T^n ,1)^* (T^n,1)) ^{1/2}.$ We write  $$H^{d^n} = \overline{\mbox {Range}} (T^n , 1) ^* \oplus \mbox{ ker} (T^n ,1). $$ $$H = \overline{\mbox {Range}} (T^n , 1)  \oplus \mbox{ ker} (T^n , 1) ^* .$$ If $ \mbox {dim} ( \mbox{ker} (T^n,1)) = \mbox{dim} (\mbox{ker} (T^n,1)^*)$ then there exists unitary $W : \mbox{ ker} (T^n ,
1) \rightarrow \mbox{ ker} (T^n,1) ^*.$ Thus $$ U = V \oplus W : \overline{\mbox {Range}} (T^n , 1) ^* \oplus \mbox{ ker} (T^n ,1) \rightarrow \overline{\mbox {Range}} (T^n , 1)  \oplus \mbox{ ker}(T^n , 1) ^*  $$  is unitary and  $ (T^n,1) = U((T^n ,1)^* (T^n,1)) ^{1/2}.$ Which implies $$ (T^n,1) (T^n,1) ^* = U (T^n,1) ^* (T^n,1) U ^*.$$ As a result we have if $\mbox {dim} (\mbox{ker} (T^n,1)) = \mbox{dim} (\mbox{ker} (T^n,1)^*)$ then $\Delta _{T^n} = \mbox {dim }(\overline{\mbox {Range}} (I_{H^{d^n}} - (T^n,1)^* (T^n,1)).$ For the converse if  $\Delta _{T^n} = \mbox{dim} (\overline{\mbox {Range}} (I_{H^{d^n}} - (T^n,1)^* (T^n,1))$   then  there exists unitary $U : H^{d^n} \rightarrow H$ such that $$(T^n,1) (T^n,1) ^* = U (T^n,1) ^* (T^n,1) U ^*.$$ First note that $(T^n ,1) ^* (T^n,1) | _{\overline{\mbox {Range}}(T^n,1) ^*}$ is unitarily equivalent with $(T^n,1) (T^n,1) ^* |_{\overline{\mbox {Range}}(T^n,1)}.$ Indeed, this can be seen by an argument similar to the proof of Lemma $1.4$ in \cite{UC}. Since $V 
^* \{(T^n,1) (T^n,1) ^* \} V = (T^n,1) ^* (T^n,1),$ where $V : \overline{\mbox {Range}} (T^n,1) ^* \rightarrow \overline{\mbox {Range}} (T^n,1)$ is defined as above,
 $$ (T^n,1) ^* (T^n,1)= A _1 \oplus 0 ~ \mbox{ on}~ \overline{\mbox {Range}} (T^n,1) ^* \oplus \mbox{ ker}(T^n,1)$$ and $$ (T^n,1) (T^n,1)^* = A _2 \oplus 0 ~ \mbox{ on}~ \overline{\mbox {Range}} (T^n,1)  \oplus \mbox{ ker}(T^n,1)^*.$$ Thus, $ \mbox{dim} (\overline{\mbox {Range}} (I_H - (T^n,1) (T^n,1) ^*))= \mbox{dim} (\overline{\mbox {Range}} (I_{H^{d^n}} - (T^n,1) ^* (T^n,1))) $ implies  that $\mbox{ dim} (\overline{\mbox {Range}} (I-A_1)) +\mbox{ dim} (\mbox{ker} (T^n,1)) =\mbox{dim} (\overline{\mbox {Range}} (I-A_2)) + \mbox{dim} (\mbox{ker} (T^n,1)^*).$ As a result, $\mbox{dim} (\mbox{ker} (T^n,1)) = \mbox{dim} (\mbox{ker} (T^n,1)^*).$
\end{proof}

\section{Maximal operator tuples}

\begin{definition}

Call an operator tuple $T=(T_1,\cdots,T_d)$ \textbf{ maximal}  if $ \Delta_{T^n} =
(1 + d + \cdots + d^{n-1}) \Delta _{T}$ for any positive integer
$n$.

\end{definition}

Given a Hilbert space $\mathcal{L}$, define the full Fock space
over $\mathcal{L}$ by
$$ \Gamma(\mathcal{L}) = \mathbb{C} \oplus \mathcal{L} \oplus
\mathcal{L}^{\otimes 2}\oplus \cdots \oplus \mathcal{L}^{\otimes
k} \oplus \cdots .$$

The one dimensional subspace $\mathbb{C} \oplus \{0\} \oplus
\{0\}\oplus \cdots $ is called the vacuum space  $\Omega$. The unit norm
element $(1, 0, 0, \ldots )$ is called the vacuum vector and is
denoted by $\omega$. The projection on to the vacuum space is
denoted by $E_0$.

Let $\{e_1, e_2, \ldots , e_d\}$ be an orthonormal basis of
$\mathbb{C}^d$. Then an orthonormal basis for the full tensor
product space $(\mathbb{C}^d)^{\otimes k}$ is $\{ e_{i_1} \otimes
\cdots \otimes e_{i_k} : 1 \le i_1, \ldots , i_k \le d \}$. Define
the {\em creation operator} tuple $V = (V_1,V_2, \ldots , V_d)$
on the full Fock space $ \Gamma(\mathbb{C}^d)$ by
$$V_i \xi = e_i \otimes \xi \mbox{ for } i = 1,2, \ldots ,d \mbox{ and }
\xi \in \Gamma(\mathbb{C}^d).$$
Needless to say that $e_i \otimes \omega$ is identified with $e_i$. It is easy to see that the $V_i$ are isometries with orthogonal ranges.

\begin{lemma}\label{eqn V}

The creation operator tuple $V = (V_1,\cdots,V_d)$ is maximal.

\end{lemma} \begin{proof}

  \noindent  We compute the action of $V_i^*$ on the
orthonormal basis elements: \begin{align*}
  \langle  V_i^* ( e_{i_1} \otimes
e_{i_2} \otimes \cdots \otimes e_{i_k} ) , \xi \rangle & =
\langle   e_{i_1} \otimes e_{i_2} \otimes \cdots \otimes e_{i_k}
 , V_i \xi \rangle \\
 & =  \langle   e_{i_1} \otimes e_{i_2} \otimes \cdots \otimes e_{i_k} ,
e_i \otimes \xi \rangle \\
 & =  \begin{cases}
\langle  e_{i_2} \otimes \cdots \otimes e_{i_n} , \xi \rangle &\
\textup{if}~
i_1 = i \\
0 &\ \textup{if}~  i_1 \neq i. \end{cases}. \end{align*}

Thus \begin{eqnarray} \label{Vi*} V_i^* ( e_{i_1} \otimes e_{i_2}
\otimes \cdots \otimes e_{i_k} ) =  \begin{cases}  e_{i_2} \otimes
\cdots \otimes e_{i_k} &\
\textup{if}~i_1 = i \\
0 &\ \textup{if}~ i_1 \neq i. \end{cases}  .
\end{eqnarray} Hence it follows that $\sum V_iV_i^* ( e_{i_1}
\otimes e_{i_2} \otimes \cdots \otimes e_{i_k} ) =  e_{i_1}
\otimes e_{i_2} \otimes \cdots \otimes e_{i_k}$ for any $k \ge 1$
and $1 \le i_1, \ldots ,i_k \le d.$

Now
$$\langle  V_i^* \omega , \xi \rangle  = \langle  \omega , e_i \otimes
\xi \rangle  = 0 \mbox{ for any } \xi \in \Gamma(\mathbb{C}^d)
\mbox{ and any } i.$$

Thus  $\sum V_iV_i^* ( \omega ) = 0$ and hence
$I - \sum V_iV_i^* $ is the
$1$-dimensional projection onto the vacuum space.   In a similar
vein, it is easy to see that
$$ I - \sum_{1 \le i_1, i_2, \ldots ,i_n \le d} V_{i_1} V_{i_2}
\ldots V_{i_n} V_{i_n}^* \ldots V_{i_2}^* V_{i_1}^* $$ is the
projection onto $\mathbb{C} \oplus \mathbb{C}^d \oplus
(\mathbb{C}^d)^{\otimes 2}\oplus \cdots \oplus (\mathbb{C}^d)^{\otimes n-1}$, the direct sum of $k$ particle spaces from $k=0$ to $n-1$.
This space has dimension $1 + d + \cdots +d^{n-1}$.
\end{proof}

\begin{lemma} \label{IO} If $W = (W_1, W_2,\cdots,W_d)$ is a contractive tuple consisting of isometries, then $W$ is maximal.

\begin{proof}
For this proof, we shall use the Wold decomposition of $W$ due to
Popescu. Theorem 1.3 of Popescu \cite{Pop-Isometric} says that the
Hilbert space $H$ decomposes into an orthogonal sum $H= H_0 \oplus
H_1$ such that $H_0$ and $H_1$ reduce each operator $W_i$ for
$i=1,2, \cdots,d$ and

\begin{enumerate}\item $(I-\sum_{i=1}^d W_iW_i^*)|_{H_1} =0$,
\item if $A_i = W_i|_{H_0}$, then the tuple $A$ is unitarily
    equivalent to the tuple consisting of $V_i\otimes
    I_{{\mathcal D}_{W^*}}$.
\end{enumerate}

Since the reducing subspace $H_1$ does not contribute to the
defect dimensions at all, we have the result.
\end{proof}

\end{lemma}
The next lemma shows that a maximal operator tuple restricted to a reducing subspace is also maximal.
\begin{lemma}
Let $(T_1,\cdots,T_d)$ be a maximal operator tuple and $M$ be a reducing subspace for each of $T_1,\cdots,T_d$,then the operator tuple $(T_1|_M,\cdots,T_d|_M)$ is also maximal. \end{lemma}

\begin{proof}
Let $\mbox{ dim }(H_1) = \mbox{ dim }\{ \overline{\mbox{ Range }} (I- \sum _{i=1} ^d T_i T_i^*)\} = n.$ The first defect space of the operator tuple $(T_1|_M,\cdots,T_d|_M)$ is given by \begin{eqnarray*} H_1^{'}&=& \overline{\mbox{ Range }}(I_M - \sum_{i=1} ^d(T_i|M) (T_i|M) ^*) \\&=& \overline{\mbox{ Range }} (I - \sum_{i=1} ^d T_iT_i^*)|_M.\end{eqnarray*} So $H_1^{'}  \subseteq H_1.$~ Now we know that the second defect space of the operator tuple $(T_1,\cdots,T_d)$ is given by \begin{eqnarray*} H_2 &=& H_1 \vee (T,1) A^d \\ &=& H_1 \vee T_1 (H_1) \vee \cdots \vee T_d(H_1).\end{eqnarray*} Since it is given that the operator tuple $(T_1,\cdots,T_d)$ is maximal, so $\mbox{ dim } (H_2) = (1+d)n= n+nd.$ Which implies $\mbox{ dim } \{T_i (H_1)\}= n  $ for all $i= 1,\cdots,d.$ Now the second defect space of the operator tuple $(T_1|_M,\cdots,T_d|_M)$ is given by $H_2^{'} ~=~ H_1^{'} \vee T_1|_{M} (H_1^{'})\vee \cdots \vee T_d|_{M}(H_1^{'}).$ Since $\mbox{ dim } \{T_i (H_1)\}= n  $ and $H_1^{'}  \subseteq H_1$, so 
$\mbox{ dim }\{T_i|_{M} (H_1^{'})\} = \mbox{ dim } {(H_1^{'} )}.$~
As a result $\mbox{ dim }{(H_2^{'})} = (1+d) \mbox{ dim } {(H_1^{'})}.$~ Similarly it is easy to see that $\mbox{ dim } (H_n ^{'}) = (1+d+\cdots+d^{n-1}) ~ \mbox {dim} (H_1^{'}).$ \end{proof}

We shall investigate the question of maximality of the tuple $(P_M V_1|_M,\cdots,P_M V_d|_M)$ where $M$ is either an invariant or a co-invariant subspace of $\Gamma (C^d).$
\begin{lemma}
If $M$ is an invariant subspace of~  $\Gamma (\mathbb{C}^d)$, then
the tuple $(V_1|_M,\cdots,V_d|_M)$ is maximal.
\end{lemma}
\begin{proof}
Since the tuple $(V_1|_M,\cdots,V_d|_M)$ is a contractive tuple consisting of isometries, so by lemma \ref{IO} the tuple is maximal.
\end{proof}

The situation for co-invariant subspaces is starkly different. In the rest of this section, we give an example of a co-invariant subspace such that the compression of the creation operators to this subspace is not maximal. We also give an example of a co-invariant subspace for which it is maximal.

Here we shall introduce the multi-index notation. Let $\Lambda$ denote the set $\{1,2,\cdots,d\}$ and $\Lambda ^m$ denote the $m-$ fold cartesian product of $\Lambda$ for $m \geq 1.$ Let $\tilde{\Lambda}$ denote $\bigcup_{m=0} ^ \infty \Lambda ^m , $ where $\Lambda ^0$ is just the set $\{0\}$ by convention. For $\alpha \in \tilde{\Lambda},$   $e _\alpha$ denote the vector $e_{\alpha_1} \otimes e_{\alpha _2} \otimes \cdots \otimes e_{\alpha _m}$ in the full Fock space $\Gamma (\mathbb C ^d),$ and $e_0$ is the vacuum vector $\omega .$

The following shows that there exists co-invariant subspace $M
\subseteq \Gamma (\mathbb{C}^d)$ such that the operator tuple
$(P_M V_1|_M,\cdots,P_M V_d|_M)$ is not maximal.

Let us consider the operator tuple $(R_1,\cdots,R_d)$ on $\Gamma (\mathbb{C}^d)$
where each $R_j: \Gamma(\mathbb{C}^d)\longrightarrow \Gamma(\mathbb{C}^d)$ is given by
$$R_j(\xi) = \xi \otimes e_j, ~\mbox{for } j = 1,2,\cdots,d ~\mbox{and} ~ \xi \in \Gamma (\mathbb{C}^d).$$
Note that $V_iR_j = R_jV_i$ for all $i,j \in \{1,2,\cdots,d\}.$
Fix any $j \in \{1,2,\cdots,d\}$ and consider the subspace $M ^
\perp \subseteq  \Gamma (\mathbb C ^d)$ given by  $M^\perp=
\overline{\mbox{ Range }} R_j.$ So the subspace $M ^ \perp$ is
invariant  under each $V_i.$ The orthonormal basis of $M^\perp$ is
given by $\{e_\alpha  \otimes e_j , \mbox{for all}~ \alpha \in
\tilde {\Lambda}\}.$
 Let us consider the operator tuple $ T= (P_M V_1|_M,\cdots,P_M V_d|_M).$
 The first defect space of the tuple is given by
 $H_1 = \overline{\mbox{ Range }} (P_M - \sum_{i=1} ^d P_M V_iP_MV_i ^*P_M).$~
 The operator  \begin {eqnarray*}    P_M - \sum_{i=1} ^d P_M V_iP_MV_i ^*P_M &=&
 P_M (I-\sum_{i=1} ^d V_i V_i ^*)P_M \\ &=& P_M E_0 P_M \\ &=& E_0.
  \end{eqnarray*}
  So $H_1 = \Omega$ and $\Delta _T= 1.$
  The second defect space of the tuple $T$ is given by
  $$H_2 = H_1 \vee P_M V_1P_M (H_1)\vee \cdots\vee P_M V_d P_M (H_1)
  = H_1 \vee P_M(e_1) \vee \cdots \vee P_M(e_d).$$ Now \begin {eqnarray*}  P_M(e_i)
  &=& e_i - P_{M^\perp} (e_i)\\ &=& e_i - \sum _ {\alpha \in \tilde{\Lambda} } \langle e_i,e_\alpha\otimes e_j\rangle e_\alpha\otimes e_j \\
  &=& e_i - \langle e_i , e_j \rangle e_j.\end{eqnarray*}
  So the second defect space has dimension $1+(d-1)=d.$ Hence the operator tuple $T$ is not maximal.\\
Our aim is to find a class of co-invariant subspaces
$M \subseteq \Gamma (\mathbb{C}^d)$ of the tuple $(V_1,\cdots,V_d)$ for which the tuple $(P_M V_1 |_M,\cdots,P_M V_d |_M)$ is maximal.\\
 Take an inner function $\varphi \in \Gamma (\mathbb{C}^d)$ which is given by $\varphi = \sum _{\alpha \in \tilde{\Lambda} } \lambda _\alpha e_\alpha$
 such that $\lambda _0 =0$ and $\lambda _\alpha \neq 0 $ for infinitely many $\alpha.$
 (See \cite {AP} for the definition of inner function.)
  Note since $\varphi$ is inner the multiplication operator $M_\varphi: \Gamma (\mathbb{C}^d) \longrightarrow \Gamma (\mathbb{C}^d) $
  given by $M_\varphi (\eta)= \eta \otimes \varphi$ is an isometry.\\ Consider the closed subspace $M^\perp \subseteq \Gamma (\mathbb{C}^d)$
given by $M^ \perp = \Gamma (\mathbb{C}^d) \otimes \varphi. $ The
subspace $M^\perp$ is invariant under each $V_i$ and it has
orthonormal basis given by $\{e_\alpha \otimes \varphi ; \alpha \in
\tilde{\Lambda}\}.$ Our claim is the tuple $T=(P_M
V_1|_M,\cdots,P_M V_d|_M )$ is maximal. The first defect space of
the tuple $T$ is given by $$H_1 = \overline{\mbox{ Range }} (P_M -
\sum_{i=1} ^d P_M V_iP_MV_i ^*P_M).$$ The operator
\begin{eqnarray*} P_M - \sum_{i=1} ^d P_M V_iP_MV_i ^*P_M &=& P_M
(I-\sum_{i=1} ^d V_i V_i ^*)P_M \\ &=& P_M E_0 P_M \\ &=& E_0.
\end{eqnarray*} So $\Delta _T= 1$ and $H_1 = \Omega.$ The second
defect space of the tuple $T$ is given by $$H_2 = H_1 \vee P_M
V_1P_M (H_1)\vee \cdots\vee P_M V_d P_M (H_1) = H_1 \vee P_M(e_1)
\vee \cdots \vee P_M(e_d).$$  Here \begin{eqnarray*} P_M(e_i) &=&
e_i - P_{M^\perp} (e_i)\\ &=& e_i - \sum _{\alpha \in
\tilde{\Lambda} } \langle e_i , e_\alpha \otimes \varphi \rangle
e_\alpha \otimes \varphi\\ &=& e_i - \langle e_i, \varphi \rangle
\varphi \\ &=& e_i - \lambda _i \varphi.\end{eqnarray*} The
vectors $(e_1 - \lambda _1 \varphi , \cdots, e_d - \lambda _d
\varphi)$ are linearly independent, because $\varphi$ has
infinitely mane co-ordinates non zero. As a result the second
defect space has dimension $d+1.$ The third defect space of the
tuple $T$ is given by $$H_3 = H_1 \vee P_M V_1P_M (H_1)\vee
\cdots\vee P_M V_d P_M (H_1) \vee  (P_M V_1P_M)^2 (H_1) \vee
\cdots \vee(P_M V_dP_M)^2( H_1).$$ Now,
\begin{align*}
P_M V_i P_M V_j P_M (H_1) = P_M V_i (e_j - \lambda _j \varphi) &=  P_M (e_i \otimes e_j) - P_M \lambda _j (e_i \otimes \varphi)\\ &= P_M(e_i \otimes e_j) \\ &= (e_i \otimes e_j) - P_{M^\perp} (e_i \otimes e_j)\\&=  (e_i \otimes e_j) - \sum _{\alpha \in \tilde{\Lambda} }  \langle e_i \otimes e_j , e_ \alpha \otimes \varphi \rangle e_\alpha \otimes \varphi\\& =(e_i \otimes e_j) - \langle e_i \otimes e_j , \varphi \rangle \varphi - \lambda _j (e_i \otimes \varphi)\\& = (e_i \otimes e_j)- \lambda _{ij} \varphi - \lambda _j (e_i \otimes \varphi) .
\end{align*} Again the vectors $\{e_i \otimes e_j - \lambda _{ij} \varphi - \lambda _j (e_i \otimes \varphi)    ~\textup{for}~ 1\leq i,j \leq d\}$ are linearly independent as $\varphi$ has infinitely many  non-zero co-ordinates. As a result the third defect space $H_3$ has dimension $1+d+d^2.$ In the same fashion the $n$-th defect space $H_n$ has dimension $1+d+\cdots+d^{n-1}.$

\section{Commuting tuples}

\begin{lemma}

For a commuting operator tuple $T= (T_1,\cdots,T_d)$ and any positive integer $n$,
$$ \Delta_{T^n} \le \sum_{k=0}^{n-1} \left( \begin{array}{c} k+d-1 \\ d-1 \end{array} \right) \Delta_{T}.$$

\end{lemma}

\begin{proof}

As before, for a contractive non-unital spanning set $T_1, T_2, \ldots ,T_d$, we have $\Delta_{T^2} \le (d+1) \Delta_T$. So the result holds for $n=2$. If it holds for $n-1$, then
$$ \Delta_{T^{n-1}} \le \sum_{k=0}^{n-2} \left( \begin{array}{c}
k+d-1 \\ d-1 \end{array} \right) \Delta_T. $$
Now $I - P_T^n(I) = I - P_T^{n-1} (I) + P_T^{n-1} ( I - P_T(I))$. So
\begin{eqnarray*} \Delta_{T^n} & \le & \sum_{k=0}^{n-2} \left( \begin{array}{c}
k+d-1 \\ d-1 \end{array} \right) \Delta_T + \mbox{ dim } \overline{\mbox{ Range }} P_T^{n-1} ( I - P_T(I)) \\
& \le & \sum_{k=0}^{n-2} \left( \begin{array}{c} k+d-1 \\ d-1 \end{array} \right) \Delta_T + \left( \begin{array}{c} n-1+d-1 \\ d-1 \end{array} \right) \Delta_T = \sum_{k=0}^{n-1} \left( \begin{array}{c}
k+d-1 \\ d-1 \end{array} \right) \Delta_T .\end{eqnarray*}
\end{proof}
\begin{definition}

Call a commuting operator tuple $T=(T_1,\cdots,T_d)$  \textbf{maximal} if for any positive integer $n$,
$$ \Delta_{T^n} =\sum_{k=0}^{n-1} \left( \begin{array}{c} k+d-1 \\ d-1 \end{array} \right) \Delta_{T}.$$

\end{definition}

The tuple $V = (V_1, V_2,
\ldots ,V_d)$ has a certain co-invariant subspace (i.e., the
subspace is invariant under $V_i^*$ for each $i$) that is of
special interest. To describe it, consider the permutation group $\sigma_k$ in $k$ symbols.
It has a unitary representation on the full tensor product space $\mathcal{L}^{\otimes k}$ for $k=1, 2,
\ldots$. The representation is defined on elementary tensors by
\begin{equation} \label{Upi} U_\pi(x_1\otimes x_2\otimes \dots\otimes x_k) =
x_{\pi^{-1}(1)}\otimes x_{\pi^{-1}(2)}\otimes\dots\otimes
x_{\pi^{-1}(k)}, \qquad x_i\in \mathcal{L}. \end{equation} The symmetric tensor product of $k$ copies of $\mathcal{L}$ is then the subspace of the full tensor product $\mathcal{L}^{\otimes
k}$ consisting of all vectors fixed under the representation of
the permutation group.
$$
\mathcal{L}^{\circleds \; k} = \{\xi\in \mathcal{L}^{\otimes k}:
U_{\pi}\xi = \xi, \quad\pi\in \sigma_k\}.
$$   Of course, $\mathcal{L}^{\circleds \;
1}=\mathcal{L}^{\otimes 1} = \mathcal L$. It is now natural to consider
the symmetric Fock space. It is a subspace of the full Fock space
and is defined by

$$ \Gamma_s(\mathcal{L}) = \mathbb{C} \oplus \mathcal{L} \oplus
\mathcal{L}^{\circleds \; 2}\oplus \cdots \oplus
\mathcal{L}^{\circleds \; k} \oplus \cdots .$$

Note that the vacuum vector of the full Fock space is in the
symmetric Fock space. We continue to denote by $E_0$ the
projection onto the one-dimensional space spanned by the vacuum
vector. Denote by $P_s$ the orthogonal projection onto the
subspace $ \Gamma_s(\mathcal{L})$ of $ \Gamma (\mathcal{L})$. Now
we specialize to $\mathcal L = {\mathbb C}^d$. As before, $\{ e_1,
e_2, \ldots ,e_d\}$ is an orthonormal basis. The projection $P_s$
acts on the full tensor product space $ \mathcal{L}^{\otimes k}$
by the following action on the orthonormal basis:
$$P_s (e_{i_1} \otimes e_{i_2} \otimes \cdots \otimes
e_{i_k})= \frac{1}{k!} \sum  e_{i_{\pi(1)}} \otimes e_{i_{\pi(2)}}
\otimes \cdots \otimes e_{i_{\pi(k)}}$$ where $\pi$ varies over the
permutation group $\sigma_k$. It is well known (see for example
\cite{Bhattacharyya-survey} that $\Gamma_s({\mathbb C}^d)$ is an invariant subspace
for $V_1^*, V_2^*, \ldots ,V_d^*$. Define the $d$-shift (see
\cite{sub3}) $S = (S_1, S_2, \ldots , S_d)$ on $
\Gamma_s(\mathbb{C}^d)$ by
$$S_i \xi = P_s (e_i \otimes \xi) = P_s V_i \xi \mbox{ for } i = 1,2,
\ldots ,d \mbox{ and } \xi \in \Gamma_s(\mathbb{C}^d).$$ Since $V_i$
are isometries, the $S_i$ are contractions. They are commuting operators, see \cite{Bhattacharyya-survey}.

\begin{lemma}

 The operator tuple $S= (S_1,\cdots,S_d)$  is a maximal commuting operator tuple.

\end{lemma}

\begin{proof}
A computation similar to the proof of Lemma \ref{eqn V} shows that for any $n=1,2, \ldots$, the operator
$$ I - \sum_{1 \le i_1, i_2, \ldots ,i_n \le d} S_{i_1} S_{i_2}
\ldots S_{i_n} S_{i_n}^* \ldots S_{i_2}^* S_{i_1}^* $$ is the
projection onto $\mathbb{C} \oplus \mathbb{C}^d \oplus
({\mathbb{C}^d})^{\circleds 2}\oplus \cdots \oplus ({\mathbb{C}^d})^{\circleds
n-1}$. This is so because of commutativity. This space has dimension
$\sum_{k=0}^{n-1} \left( \begin{array}{c} k+d-1 \\ d-1 \end{array} \right).$ \end{proof}

Consider the Arveson space $H_{d} ^2$ on the unit ball $\mathbb B _d$  defined by the reproducing kernel $K_{\lambda} (z) = 1/(1-<z,\lambda>),$ where $<z,\lambda> = \sum _{j=1} ^d z_j \overline{\lambda _j}. $ From proposition 2.13 of \cite{sub3} we know that, the spaces  $H_{d} ^2$ and $\Gamma _s(\mathbb{C}^d)$ are unitarily equivalent and the $d$ tuple of operators $(S_1,\cdots,S_d)$ on $\Gamma _s(\mathbb{C}^d)$ is unitarily equivalent to the $d$- shift $(M_{z_1} ,\cdots,M_{z_d})$ on $H_{d} ^2.$  By a multiplier of $H_{d} ^2$ we mean a complex-valued function $f$ on $\mathbb B _d$ with the property $f H_{d} ^2 \subseteq H_{d} ^2. $The set of multipliers is a complex algebra of functions defined on the ball $\mathbb B _d$ which contains the constant functions,and since $H_{d} ^2$ itself contains the constant function $1,$ it follows that every multiplier must belong to $H_{d} ^2.$ In particular, multipliers are analytic functions on $\mathbb B _d.$ It is easy to see that a multiplication operator $M_f$ on $H_
{d} ^2$ defined by a multiplier $f$ is bounded.
 Let $M \subseteq H_{d} ^2$ be  closed and invariant under action of the $d$- shift and  $T_M = (M_{z_1} |_M,\cdots,M_{z_d} |_M).$ From proposition $6.3.1$ of \cite{Chen} (page no:$142$) we know that $f M \subseteq M$, if $f$ is a multiplier of $H_{d} ^2.$ It is obvious that the $d$ - tuple $T_M$ is a $d$-contraction and hence $$M_{z_1} P_M M_{z_1}^* + \cdots+M_{z_d} P_M M_{z_d}^* \leq P_M,$$ where $P_M$ is the orthogonal projection from $H_{d} ^2$ onto $M.$ The first defect dimension for the tuple $T_M$ is given by $$\Delta _{T_M} = \mbox{dim}\{ \overline{\mbox{Range}} (P_M - \sum _{i=1} ^d M_{z_i} P_M M_{z_i}^*) \}$$
The next theorem shows that $\Delta _{T_M} \geq 2,$ for any proper closed subspace $M$ of $H_{d} ^2,$ where $d \geq 2.$

\begin{theorem}
Let $M$ be a proper, closed subspace of $H_{d} ^2$,$d \geq 2,$ which is invariant under the action of the $d-$ shift. Then  $\Delta _{T_M} \geq 2.$
\end{theorem}
\begin{proof}
Clearly $\Delta _{T_M}\geq 1.$ If $\Delta _{T_M} = 1,$ then from proposition $6.3.7$ of \cite{Chen} (page no:$148$) we have there exists a multiplier $\varphi$ of $H_{d} ^2$ such that $P_M = M_{\varphi} M_{\phi} ^* .$ Note that $\mbox {ker} (M_{\varphi}) = \{0\}.$ Indeed if $f \in  \mbox {ker} (M_{\varphi}),$ then   $$\varphi(z) f(z) = 0$$ for all $z \in \mathbb B_{d}.$ Since $\varphi$ is not identically equal to zero, there exists $w \in \mathbb B _d$ such that $\varphi (w) \neq 0.$ Which implies there exists $B (w , r)$ an open ball of radius $r$ on which $\varphi$ is not equal to zero. So $f (z) = 0$ for all $z \in B(w,r).$ As $f \in H_{d} ^2,$ so $f$ is identically equal to zero. Thus, $M = \varphi H_{d} ^2$ and $M _{\varphi}^* M _{\varphi}  = \mbox{ Id}. $ Now by Corollary $6.2.5,$ of \cite{Chen} (page no: 140), we see that $\varphi$ is constant, and hence $M = H_{d}^2.$ This contradiction shows that $\Delta _{T_M}\geq 2.$
\end{proof}

\section{Pure operator tuples}
The operator tuples $V=( V_1, V_2, \ldots ,V_d)$ on $\Gamma (C^d)$ and
 $S = (S_1, S_2, \ldots
,S_d)$ on $\Gamma _s (C^d)$ have the special
property that $ P_{V}^n(I) $ and $ P_{S}^n(I) $ converge
strongly to $0$ and hence they are \textbf{pure}.

\begin{lemma}
Let $T = (T_1,\cdots, T_d)$ be a pure operator tuple on $H$ and $M \subseteq H$ either an invariant or a co-invariant subspace for each $T _i$ and let $A_i = P_M T_i|_M$ for $i = 1,\cdots,d.$ Then the tuple $A = (A_1, \cdots,A_d)$ is pure.
\end{lemma}
\begin{proof}
First let $M \subseteq H$ is an invariant subspace for each $T_i$ and consider the tuple $A= (T_1|_M,\cdots, T_d|_M).$ It is easy to see that $P_A^n(I) \leq P_T^n (I).$ Since $P_T^n(I)$ goes strongly to zero so also $P_A^n(I).$

Now let $M \subseteq H$ is a co-invariant subspace for each $T_i$ and consider the tuple $A = (P_M T_1|_M, \cdots,P_M T_d|_M).$

  First observe that, \begin{eqnarray*} P_A(I) &=& \sum _{i=1}^d (P_M T_i|_M)
(P_M T_i|_{ M})^* \\ &=&  \sum _{i=1}^d P_{ M}T_i T_i^*|_{ M}
 \end{eqnarray*}

 Now, for $m \in { M}$ we have,
 \begin{eqnarray*}\langle \sum _{i=1}^d P_{ M}T_i T_i^*m ,m \rangle &= &
 \sum_{i=1}^d\langle P_{ M}T_i T_i^*m ,m \rangle\\&\leq& \langle \sum_{i=1}^d T_i T_i^*m ,m \rangle\\ &\leq& \langle m ,m \rangle\,.
  \end{eqnarray*}

  That implies $$P_A(I)\; \leq \;P_T(I)|_{ M} \;\leq \; I_{ M}$$

  Similarly we can show that,

  $$P_A^n (I)\; \leq \;P_T^n(I)|_{ M} $$

But the right hand side converges strongly to zero. So, $P_A^n
(I)$ converges strongly to zero.
\end{proof}

It is certainly not true that a maximal commuting operator tuple
is necessarily pure. For example, let us consider a spherical
isometry $Z= (Z_1, Z_2, \ldots ,Z_d)$ on a Hilbert space $\mathcal
N$, i.e., $Z_i$ are commuting and $\sum Z_i^* Z_i = I_{\mathcal
N}$. If for $i=1,2, \ldots ,d$, we define $A_i = S_i \oplus Z_i$
on $\Gamma_s({\mathbb C}^d) \oplus \mathcal N$, then $P_{A^n} (I)
= P_{S^n} (I) \oplus P_{Z^n} (I) = P_{S^n} (I) \oplus I$ and hence
\begin{enumerate}
\item $\Delta_{A^n} = \Delta_{S^n}$
    for every $n=1,2, \ldots$ and

\item as $n \rightarrow \infty$, the operator $P_{A^n}
        (I)$ converges strongly to a projection.

\end{enumerate}

The next lemma connects an irreducible tuple with pure tuple.
\begin{definition}
A tuple $T = (T_1,\cdots,T_d)$ on a common Hilbert space $H$ is said to be irreducible if there exists no proper closed subspace $M \subseteq H$ which is reducing under $T_i$ for $i= 1,\cdots,d.$
\end{definition}

\begin{lemma}
If $T = (T_1,\cdots,T_d)$ is an irreducible tuple such that $\Delta _T > 0,$ then $T$ is pure
 \end{lemma}

\begin{proof}

Theorem 4.5 of \cite {Bhattacharyya-survey} states that if $T =
(T_1,\cdots,T_d)$ is any contractive tuple acting on a separable
Hilbert space $H$ and $\Delta _T = m$ (which is a non-negative
integer or $\infty$), then there is a separable Hilbert space
$\mathcal{M}$ of dimension $m,$ another separable Hilbert space
$\mathcal{N}$
with a tuple of operators $Z=(Z_1,\cdots,Z_d)$ acting on it, satisfying
 $Z_i ^* Z_j = \delta _{ij}$ for $1 \leq i ,j \leq d$ and $Z_1 Z_1 ^* + \cdots + Z_n Z_n ^* = I_\mathcal{N}$ such that :\\
$(a)$ $H$ is contained in $\hat{ H} = (\Gamma (C ^d) \otimes
\mathcal{M}) \oplus \mathcal {N}$ as a subspace and it is
co-invariant under  $A = m.V \oplus Z.$\\ $(b)$ $T$ is the
compression of $A$ to $H$ i.e,.$$T_{i_1} T_{i_2} \cdots T_{i_k} h
= P_{H} A_{i_1} A_{i_2} \cdots A_{i_k} h ~  \mbox {for every} ~
h\in H, k\geq 1 \mbox{and} 1 \leq i_1 , i_2 , \cdots, i_k \leq
n.$$ \\ $(c)$  $\hat{ H} =  \overline{span} \{ A_{i_1} A_{i_2}
\cdots A_{i_k} h    ,  \mbox{where} h\in H , k\geq 1  \mbox{and} 1
\leq i_1 , i_2 , \cdots, i_k \leq n \}$. \\ Since $(T_1 , \cdots ,
T_d)$ is an irreducible tuple on $H$, so either $H \subseteq
\Gamma (C ^d) \otimes \mathcal{M}$ or $H \subseteq \mathcal{N}.$
Now if $H \subseteq \mathcal{N}$ then $\Delta _{T} = 0. $ So $H
\subseteq  \Gamma (C ^d) \otimes \mathcal{M}$ and $T_{i} = P_{H}
(V_i \otimes I_{\mathcal{M}})|_{H}.$ As a result the tuple $T= (
T_1 , \cdots , T_d)$ is pure.
 \end{proof}

 The converse of this lemma is not true. Let us consider the operator tuple $T = (V_1 \oplus S_1, \cdots, V_d \oplus S_d)$ defined on the Hilbert space $\Gamma (C^d) \oplus \Gamma _{s} (C^d).$ Here $T$ is pure, $\Delta _{T} = 2$, but $T$ is not irreducible. \\

We end with an example of a pure operator tuple, which is not maximal.Consider the Hilbert space $H = \Gamma (\mathbb C^{d-1}) \bigoplus \mathbb C$ and let $(T_1,\cdots , T_d)$ be the operator tuple on $H$ given by $$ T_1 = \begin{pmatrix} V_1 & 0 \\ 0 & 0 \end {pmatrix}  , \cdots ,T_{d-1} = \begin{pmatrix} V_{d-1} & 0 \\ 0 & 0 \end {pmatrix} , T_d = \begin{pmatrix}0 & 0 \\ 0 & rI \end{pmatrix}.$$ where $0<r<1.$


\begin{thebibliography}{10}

\bibitem{AP} A. Arias , G. Popescu, {\em Factorization and reflexivity on fock spaces}, Integr Equat Oper Th. 23(1995),  268--286.

\bibitem{sub3} W. Arveson, {\em Subalgebras of $C^*$-algebras. III. Multivariable operator theory}, Acta Math. 181 (1998), 159--228.



\bibitem{BB}
 R.Bhatia, T. Bhattacharyya,
{\em On the joint spectral radius of commuting matrices}
Studia Math. 114 (1995),  29--38.

\bibitem{BES}
T.Bhattacharyya, J.Eschmeier and J.Sarkar, {\em On CNC commuting contractive tuples}, Proc.Indian Acad.Sci.Math.Sci. 116(2006), 299--316.

\bibitem{Bhattacharyya-survey}
T. Bhattacharyya, {\em Dilation of contractive tuples : A survey}, Surveys in analysis and operator theory (Canberra, 2001), 89-126, Proc. Centre Math. Appl.
Austral. Nat. Univ., 40, Austral. Nat. Univ., Canberra, 2002.

\bibitem{Chen}
X. Chen, K. Guo,  {\em Analytic Hilbert modules}, Chapman
Hall/CRC Research Notes in Mathematics, 433, Boca Raton, FL, 2003.




\bibitem{DC}
H.-L. Gau,  P. Y. Wu,
{\em Defect indices of powers of a contraction},
Linear Algebra Appl. 432 (2010),  2824--2833.


\bibitem{UC}
H.-L. Gau,  P. Y. Wu, {\em Unitary part of a contraction} J. Math. Anal. Appl. 366 (2010),  700--705.

\bibitem{Pop-Isometric}
G. Popescu, {\em Isometric dilations for infinite sequences of
noncommuting operators}, Trans. Amer. Math. Soc. 316 (1989),  523--536.

\bibitem{Pop-Poisson}
G. Popescu, {\em Poisson transforms on some $C^*$-algebras
generated by isometries}, J. Funct. Anal. 161(1999), 27--61.





\end{thebibliography}
\end{document}